\documentclass[a4paper, 12pt]{article}
\usepackage{tgheros}
\usepackage{amsthm}
\newtheorem{thm}{Theorem}
\newtheorem{cor}[thm]{Corollary}

\newtheorem{lemma}[thm]{Lemma}

\usepackage{amssymb,amsmath}

\newtheorem*{theorem*}{Theorem}

\DeclareMathOperator{\F}{\mathbb{F}}

\begin{document}
\baselineskip=16.3pt
\parskip=14pt

\begin{center}
\section*{Realizing cyclic linear transformations as Frobenius
elements in the Galois groups of $q$-polynomials over function fields}

{\large 
Rod Gow and Gary McGuire
 \\ { \ }\\
School of Mathematics and Statistics\\
University College Dublin\\
Ireland}
\end{center}

 \subsection*{Abstract}
 
We  realize Frobenius conjugacy classes   in Galois groups 
of certain  $q$-polynomials over $\F_q(t)$
using specific degree 1 ideals. 
We combine this with 
methods from elementary linear algebra and group theory
 to realize transvections in some linear Galois groups. This enables the Galois group
 to be identified as a known classical group in several reasonably general cases.
 
 MSC 11G09, 11R58
 
 Keywords Galois group, linearized polynomial, symplectic  group
 
 \newpage
 
\section{Introduction}
  
Let $p$ be a prime number, and let $q$ be a positive integer power of $p$.
Let $\F_q$ denote the finite field with $q$ elements and let $\F_q(t)$ denote the function field
over $\F_q$ in the single variable $t$. 
A $q$-polynomial over $\F_q(t)$  is a polynomial of the form 
\begin{equation}\label{lin1}
L(x)=a_n(t) x^{q^n}+a_{n-1}(t)x^{q^{n-1}}+\cdots +a_1(t)x^q+a_0(t)x \in \F_q(t)[x].
\end{equation}
If $a_n(t)\not=0$ we say that $n$ is the $q$-degree of $f$. In this case, we usually assume that
$a_n(t)=1$, so that $L$ is monic.

The set of roots in a splitting field of the $q$-polynomial forms an $\F_q$-vector space, called
the space of roots,
as the following well known result (see \cite{SL} for example) explains.

\begin{lemma} \label{vector_space}
Let $L$ be a $q$-polynomial of $q$-degree $n$ in $\F_q(t)[x]$, with
\[
L=x^{q^n}+a_{n-1}(t)x^{q^{n-1}}+\cdots +a_1(t)x^q+a_0(t)x.
\]
Let $E$ be a splitting field for $L$ over $\F_q(t)$ and let $V$ be the set of roots of $L$ in $E$. 
Let $G$ be the Galois group of $E$
over $F$. Suppose that $a_0(t)\neq 0$. 
Then $V$ is an $\F_q$-vector space of dimension $n$ and $G$ is naturally a subgroup
of $GL(n,q)$.
\end{lemma}

The finite-dimensional $\F_q$-vector space $V$ is inside an infinite-dimensional field extension $E$ of $\F_q$,
and it is not readily apparent how to apply methods from linear algebra, such as the primary
decomposition theorem for linear transformations. In this article we propose a method for applying 
this theory to $V$. We will use these techniques in conjunction with the basic
theory of  Frobenius conjugacy classes in the Galois group of a $q$-polynomial 
over $\F_q(t)$. Frobenius classes are well known in the theory of extension fields
of global fields,  such as $\F_q(t)$ and algebraic number fields. 
We summarize the necessary background in Section \ref{backg}.
In that section we also review relevant material from 
the theory of self-reciprocal polynomials that we will use.

The Chebotarev Density Theorem concerns the density of prime ideals that give rise to a specific
conjugacy class. 
The degree (or norm) of the ideals is not specified, although there is a refined version for function fields  in 
Rosen \cite{R} which implies that a given conjugacy 
class will appear from ideals of every sufficiently large norm.
The Chebotarev Density Theorem does not give any information about Frobenius
classes arising from ideals of degree 1.

Our Frobenius class
is obtained from ideals of degree 1,
by specializing the parameter $t$ to an element of $\F_q$ to obtain a $q$-polynomial
in $\F_q[x]$ with no repeated roots (when this is possible --
all specializations in $\F_q$ might  lead to constant term 0). 
One key point to observe is that the conventional associate of this $q$-polynomial
 is the minimal polynomial of any element in the Frobenius class, see Theorem  \ref{minimal_polynomial}. 
Thus, we obtain an element of the Galois group of $L$ that acts as a cyclic element
with \emph{known} minimal polynomial 
on the space of roots of $L$.  

Use of these Frobenius classes can help to elucidate properties of the Galois group, without necessarily knowing the group in detail. It may for instance be possible to show that
some power of an element in a Frobenius class is a transvection, although
a transvection itself will virtually never arise by this degree 1 specialization process.  
This indeed is our main working method, as illustrated by
the following theorem proved in Section \ref{symp_sect}.

\begin{thm}
Let $E$ be a Galois  extension of $\F_q(t)$ which is  the splitting field of $L(x)$,
where $L(x)$ satisfies certain hypotheses (see 
Theorem $\ref{symplectic_group_as_Galois_group}$ for the precise statement).
Let $G$ be the Galois group of $E/\F_q(t)$.
Then there is a   prime ideal of degree $1$ in $\F_q[t]$
whose corresponding Frobenius class in G contains an element $\sigma$
with the property that a power of $\sigma$ is a symplectic transvection.
\end{thm}

As a corollary we are able to give an alternative proof of the well known result, that the 
symplectic group $Sp(2m,q)$ is a Galois group over $\F_q(t)$. 
In fact we provide a large class of palindromic polynomials with this Galois group.
The details are in Section \ref{symp_sect}.

We prove similar results in the context of the nonsplit orthogonal group $O^- (2m,q)$ in Section \ref{orth_sect},
finding orthogonal transvections when $q$ is even, and reflections when $q$ is odd.
The details are in Theorems  \ref{orthogonal_group_in_characteristic_2}
and \ref{orthogonal_group_in_odd_characteristic}.

In the last section of the paper, we develop a related theme, namely the idea of a conventional associate
of the $q$-palindromic polynomial $\Phi(x)+tx^{q^m}$. This is a self-reciprocal polynomial
of degree $2m$ in $\F_q(t)[x]$. Its Galois group is a subgroup of the Weyl group
of the symplectic group. 
We show that our Galois group contains a normal elementary abelian
subgroup of order $2^m$ and it is a wreath product of the form $Z_2\wr H$, where $H$ is a
transitive subgroup of the symmetric group $S_m$. We construct examples where
$H$ is $S_m$ or $A_m$, and we show other possibilities.

 Section V of \cite{MM}, together with its references, gives a good account of the problem of realizing finite groups of Lie type as Galois groups over function fields of prime characteristic. The authors use what they call additive polynomials, also
 known as vectorial polynomials, which are types of the $q$-polynomials that we study. They use the method of Frobenius
modules, a technique originally developed by Matzat. The most complete application of this technique
to realize virtually all Chevalley groups and their twisted analogues is found in the
paper by Albert and Maier, \cite{AM}. 
No powerful classification theorems of finite group theory are invoked, however quite
complex calculations with generators are necessary from time to time. Section V of \cite{MM} is mainly
an exposition of some of the results of \cite{AM}, although it includes updates
and improvements. 

The main alternative method employs Abhyankar's philosophy of \lq\lq nice polynomials
for nice groups\rq\rq. The basic idea is to start with some $q$-polynomial that one suspects
may have the properties one wants and then by various means determine the Galois group. This
may seem hit and miss, but experience and experimentation often lead to sensible examples
to use. Abhyankar and his fellow researchers frequently made use of sophisticated
classification theorems in the theory of permutation groups, and occasionally the proofs
of certain subsidiary theorems were found to be inadequate or at least difficult to follow.
Most of these problems have been resolved, but sometimes several years after their first use.
A selection of Abhyankar's numerous papers on the subject can be found in the bibliography
of \cite{MM}. We mention the paper \cite{AI} as a good exemplar of the methods used.

One may ask if any more needs to be said on the subject of realizing, say, finite classical
groups as Galois groups in their own natural characteristic, in view of the encyclopaedic
treatment already available. We feel that we have an alternative approach, not least
because our main tool is the basic theory of linear transformations applied in the context
of Galois group actions on finite vector spaces.
We can view the $q$-polynomial $\Phi(x)$ in $\F_q[x]$ used to obtain our polynomials
in $\F_q(t)[x]$ as a parameter, and as it runs through its various appropriate values,
we obtain a family of realizations of a certain classical group as a Galois group,
with the different $\Phi$ defining a large range of Frobenius classes with different minimal
polynomials.

 \section{Background}\label{backg}
 
\subsection{Galois theory relating to ring extensions}
 
 \noindent  Let $L$ be a $q$-polynomial of $q$-degree $n$ over $\F_q(t)$.
 We 
 specialize the indeterminate $t$ to some value in $\F_q$ and then examine the properties
 of the resulting polynomial. (It is possible to specialize $t$ to values
 in extension fields of $\F_q$ but this works less satisfactorily for our purpose.) This section is devoted to understanding the outcome of this process, and it relies on the exposition given by Lang, \cite{L}, pp.340-346. (We remark that the proofs given by Lang are apparently flawed with respect
 to dealing with problems of inseparability but there is no problem for us, as the residue fields
 we use are finite.)
 
 Let $I=I_\mu$ be the prime ideal in $\F_q[t]$ generated by the linear polynomial $t-\mu$,
 where $\mu\in \F_q$. Let $R$ be the integral closure of $\F_q[t]$ in the splitting field
 $E$ over $\F_q(t)$. There are only finitely many prime ideals of $R$ lying above $I$,
 and the Galois group $G$ permutes these ideals transitively in its action on $R$. Thus, let
 $J$ be a prime ideal of $R$ lying over $I$ and let $D=D_J$ be the subgroup of $G$ that fixes $J$. We call $D$ the decomposition group connected to this configuration. If we choose
 a different prime ideal lying above $I$, we obtain another decomposition group
 but it is conjugate to $D$ on account of the transitive $G$-action on prime ideals.
 Thus we have a uniquely determined $G$-conjugacy class of decomposition groups,
 and we have chosen one member of the class.
 
 We consider the quotient ring $\overline{R}=R/J$. This is a finite field extension of
 $\F_q[t]/I=\F_q$. Given $z$ in $R$, we let $\overline{z}=z+J\in \overline{R}$. The map
 sending $z$ to $\overline{z}$ is a ring homomorphism from $R$ onto $\overline{R}$, and it is
 also $\F_q$-linear. We call this the reduction homomorphism. We apply this reduction
 homomorphism to the coefficients of $L$ to obtain a polynomial $\overline{L}$ in
 $\F_q[x]$. Given the assumption that $L$ is monic of $q$-degree $n$, $\overline{L}$
 is also a monic $q$-polynomial of $q$-degree $n$ in $\F_q[x]$. 
 
 There is an induced action of $D$ on $\overline{R}$, defined as follows. Given $\sigma\in D$,
 define $\overline{\sigma}$ acting on $\overline{R}$ by
 \[
 \overline{\sigma}(\overline{z})=\overline{\sigma(z)}
 \]
 for $z$ in $R$. This is well defined, since $D$ fixes $J$ as a set. The map sending $\sigma$ to
 $\overline{\sigma}$ is a homomorphism from $D$ onto a subgroup $\overline{D}$ of $\F_q$-automorphisms of the finite field $\overline{R}$. Provided that $\overline{L}$ has
 no repeated roots (equivalently, provided that the $x$ term of $\overline{L}$ is nonzero), $D$ is isomorphic to $\overline{D}$
 and $\overline{D}$ is the full group of $\F_q$-automorphisms of $\overline{R}$. Thus, in this unramified case, since $\overline{D}$ is cyclic, $D$ is also cyclic, and we may choose
 a generator $\sigma$ of $D$ so that
\[
\overline{\sigma}(\overline{z})=(\overline{z})^q
\]
for all $z$ in $R$.

Now, in the context of $q$-polynomials, we have an enhanced structure which we intend to exploit
to obtain more information about the action of $\sigma$
on the space of roots of $L$. Recall that we have a vector space of roots of $L$ in $E$.
More specifically, $V$ is an $n$-dimensional $\F_q$-vector space contained in $R$, and $D$ acts
as $\F_q$-linear automorphisms of $V$. Under the assumption that the $x$ term of
$\overline{L}$ is nonzero, we also have a vector space $\overline{V}$ of roots of
$\overline{L}$, which is also $n$-dimensional over $\F_q$. Note that given $\alpha$ in $V$,
$\overline{\alpha}$ is in $\overline{V}$. 

\begin{lemma} \label{isomorphism_of_spaces_of_roots}
 The map $\phi:V\to \overline{V}$ given by
 \[
 \phi(\alpha)=\overline{\alpha}
 \]
 is an $\F_q$-linear isomorphism. 
 \end{lemma}
 
 This is clear, since the reduction map is $\F_q$-linear and the two spaces have the same dimension, given the assumption on the $x$ term in $\overline{L}$.
 
 Additionally, we have an action of the decomposition group $D$ on $V$ and of $\overline{D}$ on $\overline{V}$. These actions
 are $\F_q$-isomorphic, as we now show.
 
 \begin{lemma} \label{isomorphism_of_group_actions}
 The map $\phi:V\to \overline{V}$ 
 is an $\F_q$-linear isomorphism between the actions of $D$ and $\overline{D}$ on the respective
 vector spaces.
 \end{lemma}
 
 \begin{proof}
 We have by definition 
  \[
 \overline{\sigma}(\overline{\alpha})=\overline{\sigma(\alpha)}
 \]
 for $\sigma$ in $D$ and $\alpha$ in $V$. Thus,
 \[
 \phi(\sigma\alpha)=\overline{\sigma(\alpha)}=\overline{\sigma}(\overline{\alpha})=
 \overline{\sigma}\phi(\alpha)
 \]
 and this is what we require for an $\F_q$-isomorphism of actions.
 \end{proof}

 \begin{cor} \label{minimal_polynomials}
With the notation as before, suppose that the $x$ term of the reduced polynomial $\overline{L}$ is nonzero. Then the minimal polynomial of the element $\sigma$ in $G$ (defined with
respect to the degree $1$ ideal $I$) acting on the space of roots $V$ of $L$ is the same as the minimal polynomial of the Frobenius $q$-power map acting on the space of roots $\overline{V}$
of $\overline{L}$.
\end{cor}

Thus although $\sigma$ need not act as a $q$-power mapping on $V$, we can determine its minimal
polynomial by looking at the action of $\overline{\sigma}$ on $\overline{V}$, which is a $q$-power mapping.

What we think is the main significance of this explanation is that, in principle, it  is straightforward
to find the minimal polynomial of the $q$-power mapping on the space of roots of a $q$-polynomial
in $\F_q[x]$, as the next result reveals.

\begin{thm} \label{minimal_polynomial}
Let $P$ be a monic $q$-polynomial of $q$-degree $n$ in $\F_q[x]$ with
\[
P(x)= x^{q^n}+a_{n-1}x^{q^{n-1}}+\cdots+a_1x^q +a_0 x,
\]
where $a_0\neq 0$. Let $E$ be a splitting field for $P$ over $\F_q$ and let $W$
be the space of roots of $P$ in $E$. Let $F$ be the Frobenius $q$-power mapping of $W$ into itself. Then the minimal polynomial $m(x)$ of $F$ acting on $W$ is
\[
x^n+a_{n-1}x^{n-1}+\cdots +a_1x+a_0.
\]
Moreover, $W$ is cyclic with respect to the action of $F$, so that there is a root
$\alpha$, say, of $P$ such that $\alpha$, $F(\alpha)$, \dots, $F^{n-1}(\alpha)$
are an $\F_q$-basis of $W$.
\end{thm}

\begin{proof}
Let $\alpha$ be an element of $W$. Then we have $F^i(\alpha)=\alpha^{q^i}$ for all
positive integers $i$. Since $P(\alpha)=0$, it is clear that
\[
m(F)(\alpha)=0,
\]
so that $m(F)$ is the zero transformation of $W$. 

Suppose if possible that there is a monic polynomial $g$, say, in $\F_q[x]$ of degree $d$ less than
$n$ such that $g(F)=0$ on $W$. Then if we set
\[
g(x)=x^d+b_{d-1}x^{d-1}+\cdots +b_0,
\]
we find that the elements of $W$ are roots of the $q$-polynomial
\[
Q(x)=x^{q^d}+b_{d-1}x^{q^{d-1}}+\cdots +b_0 x,
\]
of $q$-degree $d$. But $Q$ has at most $q^d$ roots in its splitting field, whereas we have
$q^n$ roots in $W$ which are also roots of $Q$. This is a contradiction
and hence $m(x)$ is the minimal polynomial of $F$.

Finally, concerning cyclicity, the minimal polynomial of $F$ acting on $W$ has degree $n=\dim W$.
It is a basic theorem of linear algebra, see \cite{J}, Exercise 7, p.202, for example, that $W$ is cyclic with respect to $F$. (The result 
is a simple consequence of the theory of the rational canonical form.) This implies in addition that $m(x)$ is also the characteristic polynomial
of $F$.
\end{proof}

The appearance of the polynomial $m(x)$ in this description is well known in the theory
of $q$-polynomials over $\F_q$, as described for example in \cite{LN}, Chapter 3, Section 3, where $m(x)$
is called the conventional associate of the $q$-polynomial. Much of this theory was developed by Ore in the 1930's. In our opinion, the interpretation of the conventional associate
as the minimal polynomial of the Frobenius mapping makes more sense and enables various
theorems relating to $q$-polynomials to be proved by simple principles of the theory of linear transformations.

\begin{cor} \label{cyclic_element}
Let
\[
P(x)=x^{q^n}+a_{n-1}(t)x^{q^{n-1}}+\cdots +a_1(t)x^q+a_0(t)x,
\]
be a monic $q$-polynomial of $q$-degree $n$ whose coefficients $a_i(t)$ are elements
of the polynomial ring $\F_q[t]$. Let $G$ be the Galois group of $P$ over $\F_q(t)$.
Suppose that for some element $\lambda$ of $\F_q$, $a_0(\lambda)\neq 0$.  Then there is an element in $G$ that acts as a cyclic linear
transformation on the space $V$ of roots of $P$ whose minimal polynomial for this action is 
\[
x^n+a_{n-1}(\lambda)x^{n-1}+\cdots + a_1(\lambda)x+a_0(\lambda).
\]
\end{cor}

We call the conjugacy class of the Galois group $G$ determined by the specialization
of $t$ to $\lambda$ the corresponding Frobenius class and we also call an element of this 
class a Frobenius element. In many cases, there are $q$ different Frobenius classes, determined
by the elements of $\F_q$, but generally in this paper, we make use of one particular
Frobenius class that we can exploit most effectively.

\subsection{Some results from linear algebra}

\noindent To prove our main theorem, we require some results from elementary linear algebra, most of which we shall quote without proof. We start
with what is called the primary decomposition of a vector space with respect to a linear
transformation. See, for example, \cite{J}, Theorem 3.11, p.191.

\begin{lemma} \label{primary_decomposition}
Let $W$ be a finite dimensional vector space over $\F_q$ and let $\sigma$ be an $\F_q$-linear
transformation of $W$ into itself. Let $m(x)$ be the minimal polynomial of $\sigma$ acting on $W$. Let
\[
m(x)=p_1(x)^{r_1}\cdots p_t(x)^{r_t}
\]
be the factorization of $m(x)$ into powers of different irreducible monic polynomials
$p_i(x)$, with multiplicity $r_i$ for $1\leq i\leq t$. Then there is a unique decomposition
\[
W=W_1\oplus \cdots \oplus W_t,
\]
of $W$ into $\sigma$-invariant subspaces $W_i$, where the minimal polynomial of $\sigma$ acting
on $W_i$ is $p_i(x)^{r_i}$, $1\leq i\leq t$.
\end{lemma}

Note that, along similar lines, if we have a factorization $m(x)=m_1(x)m_2(x)$ of the minimal
polynomial of $\sigma$, where $m_1$ and $m_2$ are relatively prime, there is a corresponding
decomposition
\[
W=W_1\oplus W_2
\]
of $W$ into $\sigma$-invariant subspaces $W_i$, where the minimal polynomial
of $\sigma$ acting on $W_i$ is $m_i$.

The following result is well known, but we provide a reference for the proof.

\begin{lemma} \label{order_prime_to_p}
Let $W$ be a finite dimensional vector space over $\F_q$ and let $\sigma$ be an $\F_q$-linear
transformation of $W$ into itself. Suppose that $\sigma$ acts on $W$ as a cyclic
transformation, with minimal polynomial $m(x)$. Suppose that
\[
m(x)=p_1(x)\cdots p_t(x)
\]
where 
$p_i(x)$ is a monic irreducible polynomial of degree $d_i$, for $1\leq i\leq t$, with
no $p_i$ equal to $x$. Then the order of $\sigma$ in the group of automorphisms of $W$
is a divisor of $q^d-1$, where $d$ is the least common multiple of the $d_i$. In particular
the order of $\sigma$ is relatively prime to the characteristic $p$ of $\F_q$.
\end{lemma}

\begin{proof}
This can easily be proved from first principles. We may also use Theorem 3.11 and Lemma 8.26
of \cite{LN} (note that the matrix of a cyclic linear transformation may be taken
to be the companion matrix of its characteristic polynomial).
\end{proof}

We consider further refinements of these ideas, which again are well known in the theory
of finite groups and of algebraic groups. Let $\sigma$ be an invertible linear transformation
of the finite dimensional $\F_q$-vector space $W$. Let $m(x)$ be the minimal polynomial of
$\sigma$ and write
\[
m(x)=(x-1)^rm_2(x),
\]
where $x-1$ is relatively prime to $m_2$. Suppose that $m_2$ is square-free. Let 
\[
W=W_1\oplus W_2
\]
be the corresponding of $W$ into $\sigma$-invariant subspaces $W_1$ and $W_2$. Let $\sigma_i$
be restriction of $\sigma$ to $W_i$ for each $i$. Extend $\sigma_1$ to act on the whole 
of $W$ by defining the extension to act trivially on $W_2$, and extend
$\sigma_2$ by defining the extension to act trivially on $W_1$.

Retaining this notation, we have the following result. 

\begin{lemma} \label{extension_lemma}
We have the commuting factorization
\[
\sigma=\sigma_1\sigma_2=\sigma_2\sigma_1.
\]
The order of $\sigma_1$ is a power of $p$ and the order of $\sigma_2$ is relatively prime to
$p$. Each of $\sigma_1$ and $\sigma_2$ is a power of $\sigma$.
\end{lemma}

\begin{proof}
It is clear from the definition that the extensions commute and $\sigma$ is their product. It is
similarly clear from 
its definition that $\sigma_1$ has minimal polynomial $(x-1)^r$. Thus
$(\sigma_1-I)^r=0$. Let $p^a$ be a power of $p$ that satisfies $p^a\geq r$. Then we have
\[
(\sigma_1-I)^{p^a}=(\sigma_1-I)^{r}(\sigma_1-I)^{p^a-r}=0.
\]
But as we are working in characteristic $p$, 
\[
0=(\sigma_1-I)^{p^a}=\sigma_1^{p^a}-I
\]
and thus we see that $\sigma_1$ has order dividing $p^a$.

Lemma \ref{order_prime_to_p} shows that $\sigma_2$ has order dividing $q^d-1$ for some positive integer $d$. Now since $p^a$ and $q^d-1$ are relatively prime, we can find integers
$\alpha$ and $\beta$ such that 
\[
\alpha p^a+\beta(q^d-1)=1.
\]
We then deduce that
\[
\sigma_1=\sigma_1^{\beta(q^d-1)},\quad \sigma_2=\sigma_2^{\alpha p^a}.
\]

Since $\sigma_1$ and $\sigma_2$ commute, we obtain 
\[
\sigma^{\alpha p^a}=\sigma_1^{\alpha p^a}\sigma_2^{\alpha p^a}=\sigma_2
\]
and likewise
\[
\sigma^{\beta (q^d-1)}=\sigma_1^{\beta (q^d-1)}\sigma_2^{\beta (q^d-1)}=\sigma_1.
\]
Thus we have shown that $\sigma_1$ and $\sigma_2$ are powers of $\sigma$, as required.
\end{proof}

We call the element $\sigma_1$ the unipotent or $p$-part of $\sigma$ and
$\sigma_2$ the semisimple or $p$-regular part of $\sigma$.

\subsection{Factorization of polynomials over function fields}

\noindent The following results are  well known, but for want of a convenient reference
source, we give a proof here. 

\begin{lemma} \label{number_of_irreducible_factors}
Let $K$ be a field and let $P(x,t)$ be a monic polynomial in $K(t)[x]$, whose coefficients
are elements of $K[t]$. Write $P(x,t)$ in the form
\[
\sum_{i=0}^r p_i(x)t^i,
\]
where the $p_i(x)$ are polynomials in $x$ and $p_r(x)\neq 0$. Assume that the $p_i(x)$
are relatively prime. Then, considered as a polynomial in $K(t)[x]$,
$P(x,t)$ has at most $r$ irreducible factors.
\end{lemma}

\begin{proof}
We make use of the equality $K[t][x]=K[x][t]$, in other words, polynomials in $x$ with
coefficients that are polynomials in $t$ are also polynomials in $t$ with coefficients
that are polynomials in $x$. Consider a factorization of $P$ into irreducible factors
over the field $K(t)$, say 
\[
P=f_1(x,t)\cdots f_d(x,t),
\]
where the $f_i$ are non-constant polynomials in $x$ with coefficients in $K(t)$. Since $K[t]$ is a Dedekind
domain with quotient field $K(t)$, Gauss's Lemma implies that, as the coefficients of
$P(x,t)$, considered as a polynomial in $x$, are elements of $K[t]$, the same is true of the factors $f_i(x,t)$. 

We claim that no $f_i(x,t)$ is a polynomial in $x$ with coefficients in $K$. For if
say $f_i(x,t)=g(x)\in K[x]$, then $g(x)$ divides all the coefficient polynomials $p_i(x)$,
contrary to the relatively prime hypothesis. Thus each $f_i(x,t)$ involves a coefficient
that is a nonconstant polynomial in $t$.

Now since we have
\[
P(x,t)=
\sum_{i=0}^r p_i(x)t^i,
\]
where $p_r(x)\neq 0$, $P$ has degree $r$ when considered as a polynomial in $t$. 
Thus,
when $P$ is factored over $K(x)$, it has at most $r$ nonconstant factors. But since
$P(x,t)=f_1(x,t)\ldots f_d(x,t)$ is a factorization of $P$ into $d$ nonconstant
elements in $K(x)[t]$, it follows that $d\leq r$, as required.
\end{proof}

The following special case of the previous lemma is especially well known, and we make good use of it later.

\begin{cor} \label{irreducibility_criterion}
Let $K$ be a field and let $f$ and $g$ be relatively prime polynomials in $K[x]$. Then 
$f(x)+tg(x)$ is irreducible over the function field $K(t)$.
\end{cor}

\subsection{Construction of self-reciprocal polynomials}

\noindent When we investigate properties of classical linear groups, we encounter  so-called self-reciprocal polynomials. We devote this section to describing their properties
and constructions.

Given a monic polynomial $g$ of degree $d$ in $\F_q[x]$ with $g(0)\neq 0$, we set
\[
g^*(x)=x^dg(0)^{-1}g(x^{-1}).
\]
We note that $g^*$ is also monic of degree $d$ and it is irreducible if
$g$ is itself irreducible. We say that $g$ is self-reciprocal
if $g=g^*$. We always assume that $g(0)\neq 0$ when discussing self-reciprocal polynomials.

Let $g$ be a monic self-reciprocal polynomial of even degree $2m$ in $\F_q[x]$. Write
\[
g(x)=\sum_{i=0}^{2m}a_ix^i,
\]
where $a_{2m}=1$ and $a_0=g(0)\neq 0$. Then it is straightforward to see that $g(0)^2=1$. 

If $g(0)=1$, then $a_{m-i}=a_i$ for $0\leq i\leq m$. This situation always obtains
in characteristic 2. Suppose that $q$ is odd and $g(0)=-1$. Then we have
\[
a_{m-i}=-a_i
\]
for $0\leq i\leq 2m$. Furthermore, since $g^*(x)=g(x)$, 
\[
g(x)=-x^{2m}g(x^{-1}).
\]
Then if we set $x=\pm 1$, we see that $g(1)=-g(1)$ and $g(-1)=-g(-1)$. Thus both
1 and $-1$ are roots of $g$.

Now let $f(x)$ be a monic polynomial
of degree $m\geq 1$ in $\F_q[x]$. Then we easily verify that $h(x)$ defined by
\[
h(x)=x^mf(x+x^{-1})
\]
is a monic self-reciprocal polynomial of degree $2m$ with $h(0)=1$. Conversely, suppose that $g(x)$ is a self-reciprocal polynomial polynomial of even degree $2m$ in $\F_q[x]$
with $g(0)=1$. Then we can write
\[
x^{-m}g(x)=\sum_{i=1}^m a_i(x^i+x^{-i}),
\]
where the coefficients $a_i$ are in $\F_q$ and $a_m=1$. 

We claim that 
\[
\sum_{i=1}^m a_i(x^i+x^{-i})
\]
is expressible as $P(x+x^{-1})$ for some (in fact unique) choice of polynomial $P$ of degree $m$. Proceeding 
by induction, it suffices to prove that $x^{m}+x^{-m}$ is expressible in this form.

Now we have
\[
x^m+x^{-m}=(x^{m-1}+x^{-(m-1)})(x+x^{-1})-(x^{m-2}+x^{-(m-2)}).
\]
By induction, $x^{m-1}+x^{-(m-1)}$ and $x^{m-2}+x^{-(m-2)}$ are polynomials in $x+x^{-1}$, and it follows from the equality above that the same holds for $x^m+x^{-m}$. This proves our claim.

Note that this result does not hold if $q$ is odd and $g(0)=-1$. 

We record next some elementary facts about roots and their multiplicities. With regard to
definitions, we note that a root of multiplicity one is also called a simple root, whereas one of multiplicity greater than one is called a multiple root. Furthermore, since finite fields
are perfect, a polynomial in $\F_q[x]$ has no repeated (irreducible) factors if and only all its roots in some splitting field over $\F_q$  are simple.

The following criterion for multiple roots, based on derivatives, is well known and we omit the proof.

\begin{lemma} \label{multiple_roots}
Let $q$ be any prime power.
Let $f$ be a nonconstant polynomial in $\F_q[x]$ and let $\alpha$ be a root of $f$ in some
extension field of $\F_q$. Then $\alpha$ is a multiple root of $f$ if and only if
$f(\alpha)=f'(\alpha)=0$, where $f'$ denotes the derivative of $f$. If $\alpha$ has multiplicity
at least three, $f(\alpha)=f'(\alpha)=f''(\alpha)=0$.
\end{lemma}

We proceed to investigate factors of self-reciprocal polynomials.

\begin{lemma} \label{characteristic_2_self_reciprocal}
Let $q$ be even and let $f$ be a monic polynomial in $\F_q[x]$ of degree $m\geq 1$.
Let $h(x)=x^mf(x+x^{-1})$. Then $1$ is a root of $h$ if and only if $f(0)=0$ and this root of $h$
has multiplicity two if and only $0$ has multiplicity one as a root of $f$. All  roots
of $h$ different from $1$ are simple if and only if all  roots of $f$ different from $0$ are simple.
\end{lemma}

\begin{proof}
Since $q$ is a power of 2, 
\[
x+x^{-1}=(x+1)^2/x
\]
and thus 
\[
h(x)=x^mf((x+1)^2/x).
\]
This implies that 1 is a root of $h$ if and only if 0 is a root of $f$. Moreover, if 1 is a root
of $h$, its multiplicity is at least two, and the multiplicity is exactly two if and only
if the multiplicity of 0 as a root of $f$ is one (in other words, $x^2$ does not divide $f$).

Now let $\alpha$ be any root of $h$ different from 1. Then we see that $\alpha+\alpha^{-1}$ is
a root of $f$. Suppose that $\alpha$ is a multiple root of $h$. Then we have
$h(\alpha)=h'(\alpha)=0$. We calculate that
\[
h'(x)=mx^{m-1}f(x+x^{-1})+x^m(1+x^{-2})f'(x+x^{-1})
\]
and hence 
\[
h'(\alpha)=0=\alpha^m(1+\alpha^{-2})f'(\alpha+\alpha^{-1}).
\]
Now since $\alpha\neq 1$, we also have $1+\alpha^{-2}\neq 0$, and we deduce that
$\alpha+\alpha^{-1}$ is a multiple root of $f$. This proves what we want.
\end{proof}

\begin{cor} \label{number_of_selfreciprocal_in_char_2}
Let $q$ be even. Let $\Omega_{2m}$ be the set of 
monic self-reciprocal polynomials $h(x)$ of even degree $2m$
in $\F_q[x]$, with the property that
 $1$ is a root of $h$ with multiplicity two
and all other roots of $h$ are simple. Then $|\Omega_{2m}|$ equals
the number of monic polynomials $f$ of degree $m$ in $\F_q[x]$ with $f(0)=0$ and no repeated irreducible
factors.
\end{cor}

It is possible to evaluate $|\Omega_{2m}|$ when $q$ is a power of two, on the basis
of Corollary \ref{number_of_selfreciprocal_in_char_2}. We simply state
the relevant result without proof (the calculations hold for all $q$ but Corollary \ref{number_of_selfreciprocal_in_char_2}
only holds when $q$ is even).

\begin{thm} \label{number_of_square_free_polynomials}
Let $q$ be any prime power. Let $Z_n$ be the number of monic  polynomials $f$ of degree 
$n$ in $\F_q[x]$ with $f(0)=0$ and no repeated irreducible factors. Then  
\[
Z_n=\frac{(q^{n-1}-1)(q-1)}{q+1}
\]
when $n\geq 3$ is odd, and 
\[
Z_n=\frac{(q^{n-1}+1)(q-1)}{q+1}
\]
if $n\geq 2$ is even. $Z_1=1$.
\end{thm}



We turn to providing an analogue of Lemma \ref{characteristic_2_self_reciprocal} in odd characteristic.

\begin{lemma} \label{odd_characteristic_self_reciprocal}
Let $q$ be a power of an odd prime and let $f$ be a monic polynomial in $\F_q[x]$ of degree $m\geq 1$. Suppose that $2$ is a root of $f$, $f(-2)\neq 0$, and all roots of
$f$ are simple.
Let $h(x)=x^mf(x+x^{-1})$. Then $1$ is a root of $h$ of multiplicity two, $-1$ is not a root
of $h$,  and all other roots of $h$ have multiplicity one, 
\end{lemma}

\begin{proof}
It is clear that if $-1$ is a root of $h$, then $-2$ is a root of $f$, and since we are excluding
this possibility, we deduce that $-1$ is not a root of $h$. Let $\alpha$ be a root of $h$ different from 1. Then certainly
$\alpha+\alpha^{-1}$ is a root of $f$. Suppose if possible that $\alpha$ is a multiple root
of $h$. Then we have $h(\alpha)=h'(\alpha)=0$. 

We differentiate $h$ to find that
\[
h'(x)=mx^{m-1}f(x+x^{-1})+(x^m-x^{m-2})f'(x+x^{-1}),
\]
whence it follows that
\[
h'(\alpha)=0=(\alpha^m-\alpha^{m-2})f'(\alpha+\alpha^{-1}).
\]
Since we are assuming that $\alpha\neq 1$, and $-1$ is not a root, $\alpha^m-\alpha^{m-2}$
is not zero, and therefore $f'(\alpha+\alpha^{-1})=0$. This implies that
$\alpha+\alpha^{-1}$ is a double root of $f$, which we again excluded. Thus all
roots of $h$ other than 1 are simple. 

We turn to consideration of the multiplicity of 1 as a root of $h$. From our formula
for $h'$, we readily see that $h'(1)=0$ and deduce that 1 has multiplicity at least two
as a root of $h$. Differentiating $h'$ and setting $x$ equal to 1, we find that
\[
h''(1)=2f'(2).
\]
Thus $1$ has multiplicity at least three as a root of $h$ if and only if 2 has multiplicity
at least two as a root of $f$. Since we are assuming that all roots of $f$ are simple,
1 has multiplicity two as a root of $h$, as required.
\end{proof}

\section{The symplectic group}\label{symp_sect}

This section has two subsections. 
In the first subsection we present some structural results for the symplectic group.
In the second subsection we use these to prove our main result using degree 1
ideals to realize certain Frobenius classes.

\subsection{Structure of the symplectic group}

\noindent We use this section to recall and prove some properties of the
symplectic group that we will use in later sections. We restrict attention to finite fields,
although much of the theory applies to arbitrary fields. 

Let $W$ be a vector space of even dimension $2m$ over $\F_q$ and let $B:W\times W\to\F_q$ 
be a nondegenerate alternating bilinear form. Thus we have $B(w,w)=0$ for all $w\in W$.
We call an $\F_q$-linear automorphism $\sigma$ of $W$ an isometry (of $B$ or of $W$) if
\[
f(\sigma(u),\sigma(w))=f(u,w)
\]
for all $u$ and $w$ in $W$. The isometries of $B$ form a group under composition, called
the symplectic group associated to $B$. Since it is well known that there is
a single equivalence class of nondegenerate alternating bilinear forms defined on $W\times
W$, the symplectic group is uniquely determined up to conjugacy in the general linear group
and we denote it by $Sp(2m,q)$.

We call an element  $\sigma$ of $Sp(2m,q)$ 
a (symplectic) \emph{transvection} if its minimal polynomial is $(x-1)^2$ and $\sigma-I$ has rank one. Note
that if $q$ is a power of the prime $p$, a transvection $\sigma$ has order $p$. For if
$(\sigma-I)^2=0$, then $(\sigma-I)^p=\sigma^p-I=0$, and thus $\sigma^p=I$.

Suppose then that $\sigma$ is a transvection in $Sp(2m,q)$. It is well known that 
there are elements $u\neq 0$ in $W$ and $c\neq 0$ in $\F_q$ such that 
\[
\sigma(w)=w+cB(w,u)u
\]
for all $w\in W$. Conversely, any linear transformation $\sigma$ of $W$ defined in this way
is an element of $Sp(2m,q)$. 

Concerning the scalar $c$, we make the following observation. Suppose that $c$ is a square
in $\F_q$, say $c=d^2$. Then we set $v=du$ and we obtain
\[
\sigma(w)=w+B(w,v).
\]
In this case, we see that there is no need for the appearance of the scalar $c$ (since $v$ serves as well as $u$ for our purpose). We deduce that if $q$ is a power of 2, all symplectic
transvections are expressible by the formula
\[
\sigma(w)=w+B(w,u)u
\]
as $u$ runs over the nonzero elements of $W$. 

We consider what happens when $q$ is odd. It is easy to see that $u$ and $-u$ determine the same
symplectic transvection and this is the only way in which duplication occurs. 
Thus there are $(q^{2m}-1)/2$ symplectic transvections of this form. Now let
$c$ be a nonsquare in $\F_q$. Then symplectic transvections of the form
\[
\sigma(w)=w+cB(w,u)u
\]
cannot be expressed  as $w\to w+B(w,v)v$ for any choice of $v$. Thus, for this fixed nonsquare
$c$, we obtain a further $(q^{2m}-1)/2$ symplectic transvections and we have thus
accounted for a total of $q^{2m}-1$ symplectic transvections.

We proceed to describe the conjugacy classes of $Sp(2m,q)$ determined by transvections.

\begin{lemma} \label{conjugate_elements}
Let $\sigma$ and $\tau$ be symplectic transvections
acting on $W$, with 
\[
\sigma(w)=w+cB(w,u)u, \quad \tau(w)=w+cB(w,v)v,
\]
where $u$ and $v$ are nonzero elements of $W$ and $c$ is a nonzero scalar. Suppose
that  $g$ is an element of $Sp(2m,q)$ that satisfies $gu=v$.
Then we have
$\tau=g \sigma g^{-1}$.
\end{lemma}

\begin{proof}
 We have
\[
\tau g(w)=g(w)+cB(gw,v)v
\]
and likewise 
\[
g\sigma(w)=g(w)+cB(\sigma w,u)g(u).
\]
Thus since $gu=v$, we can complete the proof by showing that
\[
B(gw, v)=B(\sigma w, u)
\]
holds for all $w$. As $g$ is an isometry of $B$, we have
\[
B(gw,v)=B(w, g^{-1}v)=B(w,u).
\]
But $\sigma$ is also an isometry and it fixes $u$. Therefore,
\[
B(\sigma w, u)=B(\sigma w, \sigma u)=B(w,u),
\]
as required.
\end{proof}

When we take into account the fact that $Sp(2m,q)$ acts transitively on the nonzero elements
of the underlying space, and keep in mind the role of the scalar $c$ used to describe transvections,
we obtain the following result.

\begin{cor} \label{conjugacy_classes}
If $q$ is even, then the symplectic group $Sp(2m,q)$ contains
a single conjugacy class of symplectic transvections, which has size $q^{2m}-1$.

If $q$ is odd, $Sp(2m,q)$ contains exactly two conjugacy classes of symplectic transvections,
and each such class has size $(q^{2m}-1)/2$.
\end{cor}

An interesting point emerges in the proof of Lemma  \ref{conjugate_elements}. 
Let $G$ be a subgroup of $Sp(2m,q)$
that acts transitively on the nonzero elements of $W$. Then $G$ has a transitive action
on any given  conjugacy class of  transvections in $Sp(2m,q)$. Thus, if $G$ contains
a transvection, it contains the entire conjugacy class of transvections
corresponding to the given scalar $c$. 
Surprisingly, perhaps, this is enough
to prove that $G$ is equal to $Sp(2m,q)$. 
This is an idea which we will use to identify some Galois groups
in the next section. We proceed to proofs of these claims.

\begin{lemma} \label{contains_conjugacy_class}
Let $G$ be a finite group and let $H$ be a subgroup of $G$. Suppose that $H$ contains
a conjugacy class $K$ of $G$. Then $H$ contains a normal subgroup of $G$ that contains $K$.
\end{lemma}

\begin{proof}
Given $g\in G$, the subgroup $g^{-1}Hg$ of $G$ is called a $G$-conjugate of $H$. The intersection
of all the conjugates of $H$ is a subgroup of $G$ contained in $H$ and it is clearly normal
in $G$. It is called the core of $H$. Since $K$ is contained in all conjugates of $H$, as it
is invariant under $G$-conjugation, $K$ is contained in the core of $H$, as required.
\end{proof}

\begin{cor} \label{is_the_whole_symplectic_group}
Let $G$ be a subgroup of the symplectic group $Sp(2m,q)$. Suppose that $G$ acts transitively
on the nonzero elements of the underlying vector space of dimension $2m$. Suppose also
that $G$ contains a symplectic transvection. Then $G=Sp(2m,q)$. 
\end{cor}

\begin{proof}
Let $\sigma$ be  a symplectic transvection contained in $G$. The proof of Lemma \ref{conjugate_elements}
implies that $G$ contains the $Sp(2m,q)$ conjugacy class $K$, say, that contains $\sigma$. Then
Lemma \ref{contains_conjugacy_class} implies that $G$ contains a normal subgroup
of $Sp(2m,q)$ containing $K$.

 Consider the case that $q$ is even. By Corollary \ref{conjugacy_classes},
 $Sp(2m,q)$ has only one conjugacy class of transvections. 
 Therefore, $G$ contains all the transvections.
Since $Sp(2m,q)$ is generated by its transvections, 
 (see \cite{Wil} Section 3.5.1 for example) $G$ is equal to $Sp(2m,q)$. 

Now consider the case that $q$ is odd. Let $S$ be the subgroup of $Sp(2m,q)$ generated
by the transvections in a conjugacy class contained in $G$. We claim that $S$ is equal to
$Sp(2m,q)$. Certainly $S$ is normal in $Sp(2m,q)$,
as it is generated by a conjugacy class of $Sp(2m,q)$, and it is contained in $G$. 
However, apart from one exception for odd $q$, 
the only proper normal subgroup of $Sp(2m,q)$ is the central one of order two generated by
$-I$ (see for example \cite{J}, Theorem 6.16, p.396). Clearly, $S$ is not this central subgroup. 

The only possible exception arises with $Sp(2,3)$. 
This group contains two conjugacy classes
of transvections, but one class is the inverse of the other. Thus, the subgroup generated by
either class of transvections is $Sp(2,3)$, and we are finished.
\end{proof}

\subsection{Realizing symplectic transvections}

\noindent Let $\Phi$ be a monic $q$-palindromic $q$-polynomial of even $q$-degree $2m$ in $\F_q[x]$. Write
\[
\Phi=\sum_{i=0}^{2m} c_i x^{q^i},
\]
where the coefficients $c_i$ are in $\F_q$, $c_0=c_{2m}=1$, and
\[
c_{2m-i}=c_i
\]
for all $i$. We form the $q$-polynomial
\[
L=L(x)=\Phi+tx^{q^m}
\]
in $\F_q(t)[x]$.  It follows from Corollary \ref{irreducibility_criterion} that $L(x)/x$ is irreducible over $\F_q(t)$.
Let $G$ be its Galois group over $\F_q(t)$ and let $V$ be its space of roots in its splitting
field. By Galois theory, as $L(x)/x$ is irreducible, $G$ acts transitively on the nonzero
elements of $V$.

Since $\Phi$ is $q$-palindromic, so also is $L$. Thus, by \cite{E}, Section 5, there is a nondegenerate alternating bilinear form defined on $V\times V$ which is invariant under
$G$-action. Hence $G$ is a subgroup of the symplectic group $Sp(2m,q)$. 

We intend to specialize $t$ to the value 0. Then the corresponding reduced polynomial
$\overline{L}$ in $\F_q[x]$ is $\Phi$. Since $\Phi$ has no repeated roots, there is
a corresponding Frobenius class in $G$. An element in this class acts as cyclic element on
$V$ and its minimal polynomial, $M(x)$, say, is
\[
x^{2m}+c_{2m-1}x^{2m-1}+\cdots + c_1x+1.
\]
Because the coefficients $c_i$ satisfy $c_i=c_{2m-i}$ for all $i$, $M(x)$ 
is self-reciprocal.

Our intention is to show that there are many choices of $\Phi$, and corresponding $M(x)$,
that guarantee that some power of an element in the corresponding Frobenius class is a symplectic transvection. In these cases, Corollary \ref{is_the_whole_symplectic_group} will enable us to conclude that $G$ is $Sp(2m,q)$.

Here is our main result of this section.

\begin{thm} \label{symplectic_group_as_Galois_group}
Let
$M(x)$ be a monic self-reciprocal polynomial of even degree $2m$
in $\F_q[x]$, satisfying $M(0)=1$. Suppose that $x-1$ is a factor of $M(x)$ with multiplicity two 
and $M_1(x)=M(x)/(x-1)^2$ has no factor equal to $x\pm 1$ and all its
irreducible factors have
multiplicity one.
Set
\[
M(x)=x^{2m}+c_{2m-1}x^{2m-1}+\cdots +c_1 x+1,
\]
where the coefficients $c_i$ satisfy $c_i=c_{2m-i}$. Define the $q$-polynomial $\Phi$ by
\[
\Phi(x)=x^{q^{2m}}+c_{2m-1}x^{q^{2m-1}}+\cdots +c_1 x^q+x
\]
in $\F_q[x]$ and similarly define $L$ by
\[
L(x)=\Phi(x)+tx^{q^m}.
\]
 Let $G$ be the Galois group  of $L(x)$ over $\F_q(t)$.
Then there is a prime ideal of degree $1$ in $\F_q[t]$
whose corresponding Frobenius class in $G$ contains  an element a power of which is a symplectic transvection.
\end{thm}

\begin{proof}
Let $V$ be the space of roots of $L$ in its splitting field over $\F_q(t)$. As 
we mentioned earlier, since $\Phi$
is $q$-palindromic, there is a nondegenerate alternating bilinear form
defined on $V\times V$ which is $G$-invariant. 
 Thus $G$ is a subgroup 
of $Sp(2m,q)$ that acts transitively
on the nonzero elements of the underlying vector space, since as we have observed, 
$L(x)/x$ is irreducible in $\F_q(t)[x]$ .

Let $\sigma$ be an element in the  Frobenius class of $G$ defined by the specialization of $t$ to 0.
Corollary \ref{cyclic_element} shows that $\sigma$ has minimal polynomial $M(x)$ when it acts on $V$. Lemma \ref{extension_lemma}
and its
proof imply that we can write 
\[
\sigma=\sigma_1\sigma_2=\sigma_2\sigma_1,
\]
where $\sigma_1$ has minimal polynomial $(x-1)^2$ and fixes pointwise a hyperplane
of $V$, and $\sigma_2$ has order prime to $p$. Thus $\sigma_1$ is a transvection. But
Lemma \ref{extension_lemma} shows that $\sigma_1$ is a power of $\sigma$ and hence it is
in $G$. Thus $\sigma_1$ is a symplectic transvection. 
\end{proof}

\begin{cor} \label{symplectic_group_as_Galois_group2}
Let $L(x)$ be as in Theorem \ref{symplectic_group_as_Galois_group}.
Then the Galois group $G$ of $L(x)$ over $\F_q(t)$ is $Sp(2m,q)$.
\end{cor}

\begin{proof}
Theorem \ref{symplectic_group_as_Galois_group} implies that $G$
contains a symplectic transvection.
Corollary \ref{is_the_whole_symplectic_group} implies that $G$ is $Sp(2m,q)$.
\end{proof}

Note that when $m=1$, there is a unique polynomial $L=x^{q^2}+tx^q+x$, and its Galois group
is $SL(2,q)=Sp(2,q)$.

We show next that there are many choices for $\Phi$ when $q$ is a power of 2, each leading to a different Frobenius class.

\begin{thm} \label{choices_for_Phi}
Suppose that $q$ is a power of $2$. Then in the construction described by Theorem \ref{symplectic_group_as_Galois_group}, there are 
\[
\frac{(q^{m-1}-1)(q-1)}{q+1}
\]
choices for $\Phi$ if $m$ is odd, and 
\[
\frac{(q^{m-1}+1)(q-1)}{q+1}
\]
choices if $m$ is even.
\end{thm}

\begin{proof}
In characteristic 2, $\Phi$ is determined by a monic self-reciprocal polynomial $M(x)$ 
in $\F_q[x]$ with the property that $x-1$ has multiplicity two as a divisor 
of $M(x)$ and $M(x)/(x-1)^2$ is square-free. Corollary \ref{number_of_selfreciprocal_in_char_2}
shows that the number of such $M$ equals the number of monic square-free polynomials
$f$ of degree $m$ in $\F_q[x]$ with $f(0)=0$. Theorem \ref{number_of_square_free_polynomials}
in turn gives us the number of such polynomials $f$.
\end{proof}

\section{The orthogonal group of even dimension}\label{orth_sect}

\noindent Following the pattern of the previous section, we shall discuss properties
of orthogonal groups in this section.

Let $W$ be a vector space of even dimension
$2m$ over $\F_q$. Let $f:W \times W\to \F_q$ be a nondegenerate symmetric  bilinear
form. Recall that a function $Q:W\to \F_q$ is called a quadratic form with polarization $f$ if
we have $Q(\lambda w)=\lambda^2Q(w)$ for all $\lambda$ in $\F_q$ and all $w$ in $W$ and also
\[
Q(u+w)=Q(u)+f(u,w)+Q(w)
\]
for all $u$ and $w$ in $W$. 

As is well known, when $q$ is odd, $f$ and $Q$ determine each other, since $2Q(w)=f(w,w)$
for all $w\in W$. However, this is not true when $q$ is even: $Q$ determines $f$ but
$f$ does not determine $Q$. Note that when $q$ is even, the equation $2Q(w)=f(w,w)$ implies that
$f$ is alternating. 
For arbitrary $q$, we say that a quadratic form is nondegenerate if its polarization
is nondegenerate.

We say that an invertible  linear transformation $\sigma:W\to W$ is an isometry of $Q$ if we have
$Q(\sigma w)=Q(w)$ for all $w$ in $W$. It is easy to see that an isometry of $Q$ must also be an isometry of the polarization $f$. The isometries of $Q$ form a group under composition, called
the orthogonal group determined by $Q$. From what we have said above, the orthogonal group
is thus a subgroup of the symplectic group determined by $f$ when $q$ is even.

It is well known that, up to equivalence defined by the action of invertible linear transformations, there are precisely two equivalence classes of quadratic forms defined
on $W$ with nondegenerate polarization, the two equivalence classes being distinguished in the following way. We say that a vector $w$ is isotropic (with respect to $Q$) if $Q(w)=0$. 
For one equivalence class of forms, the number of nonzero isotropic vectors is
\[
(q^{m-1}+1)(q^m-1).
\]
For the other equivalence class, the number is
\[
(q^{m-1}-1)(q^m+1).
\]
The first number is larger than the second. The orthogonal group of a quadratic form with the larger number of isotropic points is denoted by $O^+(2m,q)$, the orthogonal group
for the other type being denoted by $O^-(2m,q)$. We let $O(2m,q)$ denote either of these two
types of orthogonal group (they are not isomorphic, as they have different orders,
$|O^-(2m,q)|$ being the larger order). We remark that the word singular is sometimes
used in place of isotropic in the terminology of quadratic forms.

We now divide our discussion into two cases, starting with the odd characteristic case. 
Thus, let $q$ be odd and let $Q$ be a nondegenerate quadratic form defined on $W$.
An element $\sigma$ of $O(2m,q)$ is called a reflection if $\sigma$ has order 2
and $\sigma$ fixes pointwise a subspace of codimension one in $W$. It is well known
that a reflection is expressible in the form
\[
\sigma(w)=w-f(w,u)u,
\]
where $u$ is a vector that satisfies $Q(u)=1$ (thus $f(u,u)=2$). We say that a vector
$u$ with $Q(u)=1$ is a vector of norm one and $\sigma$ is the reflection determined
by the vector $u$. Note that $\sigma(u)=-u$ and $\sigma$ fixes pointwise
the subspace of vectors orthogonal to $u$.

A consequence of the Cartan--Dieudonn\'e theorem tells us that $O(2m,q)$ is generated
by its reflections, \cite{Gr}, Theorem 6.6, p.48. Note that a reflection has determinant $-1$, so that $O(2m,q)$
has a subgroup of index 2, denoted by $SO(2m,q)$, and called the special orthogonal group.
It consists of orthogonal isometries of determinant 1.

As far as the problem of realizing subgroups of $O(2m,q)$ as Galois groups
is concerned, the following ideas are relevant
to us. Let $G$ be a subgroup of $O(2m,q)$ and let $u_1$ and $u_2$ be vectors in $W$ of norm one.
Suppose that there is an element $g$ in $G$ with $g(u_1)=u_2$. Then the reflections
determined by $u_1$ and $u_2$ are conjugate under the action of $g$. This may be proved in the spirit of Lemma \ref{conjugate_elements}. 

The following analogue of Lemma \ref{is_the_whole_symplectic_group} is then straightforward to prove given the Cartan--Dieudonn\'e theorem.

\begin{lemma} \label{acts_Transitively_on_reflections}
Let $G$ be a subgroup of $O(2m,q)$, where $q$ is odd. Suppose that $G$ acts transitively
by conjugation
on the reflections in $O(2m,q)$. Suppose also that $G$ contains a reflection.
Then $G=O(2m,q)$.
\end{lemma}
 
 We turn to a sketch of properties of the orthogonal
 group in characteristic 2. Let $q$ be even and let $Q$ be a nondegenerate quadratic form defined on $W$, with polarization $f$. Let $\sigma$ be a transvection in $O(2m,q)$.
   We know that, for all
 $w$ in $W$, we can write
 \[
 \sigma(w)=w+f(u,w)u
 \]
 for a unique $u\in W$ (since we are working in characteristic 2). Thus we have
 \[
 Q(w)=Q(\sigma(w))=Q(w+f(u,w)u).
 \]
 Since $Q$ is a quadratic form with respect to $f$, 
 \[
 Q(w+f(u,w)u)=Q(w)+f(u,w)^2+f(u,w)^2Q(u).
 \]
 It follows that
 \[
 f(u,w)^2=f(u,w)^2Q(u).
 \]
 Since this holds for all $w$, $Q(u)=1$ and we deduce that $u$ has norm one.
 
 Conversely, any symplectic transvection defined by an element of norm one is an orthogonal
 transvection. We have thus shown the following result.
 
 \begin{lemma} \label{orthogonal_transvections}
 The number of orthogonal transvections in $O(2m,q)$ equals the number of vectors of norm one
 in $W$.
 \end{lemma}
 
 In the case of $O^+(2m,q)$, this number is $q^{m-1}(q^m-1)$, while for $O^-(2m,q)$, it is
 $q^{m-1}(q^m+1)$. 
 
 It is straightforward to show that the orthogonal group acts transitively on all vectors
 $w$ satisfying $Q(w)=\lambda$ for any choice of $\lambda$ in $\F_q$. Thus, with one exception,
 $O(2m,q)$ has exactly $q$ orbits on the nonzero elements of $W$ (the exception occurs
 for $O^-(2,q)$, since there are no nonzero isotropic vectors in the corresponding geometry).
 
 When we take into account the transitive action of $O(2m,q)$ on vectors of norm one, the following analogue of Corollary \ref{conjugacy_classes} may easily be established. We omit the formal proof. 
 
 \begin{lemma} \label{conjugacy_classes_of_orthogonal_transvections}
Suppose that $q$ is even. Then there is a single conjugacy class  of transvections in $O(2m,q)$. The size of this class
is the number of vectors of norm one in the underlying vector space.
\end{lemma}

We close this section by establishing an analogue of Lemma \ref{is_the_whole_symplectic_group}.

\begin{lemma} \label{is_the_whole_orthogonal_group}
Let $q$ be a power of $2$ and let $G$ be a subgroup of $O(2m,q)$. Suppose that $G$ contains
a transvection and also acts transitively on the vectors of norm one in the underlying
space. Then $G$ is the whole of $O(2m,q)$.
\end{lemma}

\begin{proof}
We essentially duplicate the proof of Lemma \ref{is_the_whole_symplectic_group}, but use the fact
that $O(2m,q)$ is generated by its transvections. Since $G$ contains a transvection, and acts transitively on the vectors of norm one, it contains all orthogonal transvections, by the argument outlined to prove Lemma \ref{conjugacy_classes_of_orthogonal_transvections}. But the group
$O(2m,q)$ is generated by its transvections, as discussed in Section 3.8.2 and exercise 3.33
of \cite{Wil}. It follows that $G$ is $O(2m,q)$.
\end{proof}

\subsection{On $q$-anti-palindromic polynomials}

\noindent  Let $\Phi$ be a monic $q$-polynomial
of even $q$-degree $2m$ in $\F_q[x]$. Write
\[
\Phi=\sum_{i=0}^{2m}c_ix^{q^i}
\]
where the coefficients $c_i$ are in $\F_q$, and $c_{2m}=1$. 

We say that 
$\Phi$ is $q$-anti-palindromic if
 $c_i=-c_{2m-i}$ for $0\leq i\leq
m$. When $q$ is odd, this implies that $c_m=0$. When $q$ is even, we will assume
additionally that $c_m=0$. Given that $\Phi$ is $q$-anti-palindromic, this is equivalent
to saying that $\Phi(1)=0$. 



We wish to consider $q$-polynomials $L$ of $q$-degree $2m$ in $\F_q(t)[x]$ that have the form
\[
L(x)=\Phi(x)+t^qx^{q^{m+1}}-tx^{q^{m-1}},
\]
where $\Phi$ is $q$-anti-palindromic in $\F_q[x]$, as described above. We assume throughout that
$m\geq 2$. 

\begin{lemma} \label{special_polynomial}
Let $\Phi$ be a monic $q$-anti-palindromic $q$-polynomial of even $q$-degree $2m$ in
$\F_q[x]$. 
Then there is a monic polynomial $f$ of degree $q^{m-1}(q^m+1)$ in $\F_q[x]$
such that
\[
f^q-f=x^{q^m}\Phi.
\]
The polynomial $f$ is unique if we assume that $f(0)=0$. The powers of $x$ that occur
in $f$ with nonzero coefficient have the form $x^{q^i+q^j}$ for suitable non-negative
integers $i$ and $j$.
\end{lemma}

\begin{proof}
Write
\[
\Phi=\sum_{i=0}^{2m} c_i x^{q^i},
\]
where $c_{2m}=1$, $c_i=-c_{2m-i}$ for $0\leq i\leq m$, and $c_m=0$. 
The coefficient of  the monomial $x^{q^m+q^i}$ occurring in $x^{q^m}\Phi$ is the negative of the coefficient of
$x^{q^m+q^{2m-i}}$. Thus, since the mapping sending a polynomial $f$ to $f^q-f$ is
$\F_q$-linear, it suffices to show that we can solve
\[
f^q-f=x^{q^m+q^{2m-i}}-x^{q^m+q^i}
\]
for $0\leq i\leq m-1$. 

We set $X=x^{q^{m-i}+1}$. Then we need to show that we can find a polynomial $g$ in $\F_q[x]$
with
\[
g(X)^q-g(X)=X^{q^m}-X^{q^i}.
\]
We take
\[
g(X)=X^{q^{m-1}}+\cdots +X^{q^i}
\]
and may easily verify that this polynomial has the desired property. Notice then
that
\[
f(x)=x^{q^{m-1}(q^{m-i}+1)}+\cdots +x^{q^i(q^{m-i}+1)},
\]
which implies that $f$ has the stated form in terms of the powers of $x$ that occur with
nonzero coefficients. 

Finally, the kernel of the mapping sending a polynomial $h$ to $h^q-h$ is the subspace of constants and thus $f$ is determined up to the addition of a constant.
\end{proof}

\begin{cor} \label{factoring_the_polynomial}
Let $\Phi$ be a monic $q$-anti-palindromic $q$-polynomial of even $q$-degree $2m$ in
$\F_q[x]$. Let $L$ be the $q$-polynomial in $\F_q(t)[x]$ defined by 
\[
L(x)=\Phi(x)+t^qx^{q^{m+1}}-tx^{q^{m-1}}.
\]
Let $f$ be the monic polynomial that satisfies $f(0)=0$ and $f^q-f=x^{q^m}\Phi$.
Then there is a factorization
\[
L(x)=x^{-q^{m}}\prod_{\lambda\in\F_q}(tx^{q^m+q^{m-1}}+f+\lambda).
\]
Thus $L(x)/x$ has exactly $q$ different irreducible factors. One irreducible factor
has degree $(q^{m-1}-1)(q^m+1)$ and the other $q-1$ factors all have degree $q^{m-1}(q^m+1)$.
\end{cor}

\begin{proof}
We multiply $L$ by $x^{q^m}$ to obtain
\[
x^{q^m}L=t^qx^{q^{m+1}+q^m}-tx^{q^m+q^{m-1}}+x^{q^m}\Phi
\]
and then apply Lemma \ref{special_polynomial} to obtain
\[
x^{q^m}L=t^qx^{q^{m+1}+q^m}-tx^{q^m+q^{m-1}}+f^q-f.
\]
It follows that
\[
x^{q^m}L=T^q-T,
\]
where $T=tx^{q^m+q^{m-1}}+f$. Since we have 
\[
Z^q-Z=\prod_{\lambda\in\F_q}(Z+\lambda),
\]
it follows that
\[
x^{q^m}L=\prod_{\lambda\in\F_q}(tx^{q^m+q^{m-1}}+f+\lambda).
\]

Now provided that $\lambda\neq 0$, a factor $tx^{q^m+q^{m-1}}+f+\lambda$ is irreducible
by Corollary \ref{irreducibility_criterion}. Thus we have $q-1$ irreducible factors all
of degree $\deg f=q^{m-1}(q^m+1)$. Since $x^{q^m+1}$ divides $f$, $x^{-(q^m+1)}(tx^{q^m+q^{m-1}}+f)$ is also a factor, and it is
irreducible, of degree $\deg f-q^m-1$. 
\end{proof}

\begin{cor} \label{group_acts_irreducibly}
Let $L$ be as in Corollary \ref{factoring_the_polynomial} and let $G$ be the Galois group
of $L$. Then $G$ acts irreducibly on the space of roots of $L$.
\end{cor}

\begin{proof}
Let $V$ be the space of roots of $L$ and let $U$ be a nonzero $G$-invariant subspace of $V$.
$U\setminus \{0\}$ must be a union of all the roots of various irreducible factors of $L$
and hence $q^{\dim U}-1$ must be expressible as the sum of the degrees of these irreducible
factors. It is easy to see that this can only happen if $\dim U=2m$, so that $U=V$.
\end{proof}

\subsection{A nondegenerate quadratic form}

\noindent We want eventually to prove that the Galois group of the
polynomial studied in Corollary \ref{factoring_the_polynomial} is an orthogonal group.
To this end, we need to construct an invariant quadratic form, which is the content
of our next result.

\begin{thm} \label{invariant_quadratic_form}
Let  $\Phi$ be a monic 
$q$-anti-palindromic $q$-polynomial of even $q$-degree $2m$ in $\F_q[x]$. Assume that the coefficient
of $x^{q^m}$ in $\Phi$ is zero.
Let $L$ be the $q$-polynomial in $\F_q(t)[x]$ defined by
\[
L(x)=\Phi(x)+t^qx^{q^{m+1}}-tx^{q^{m-1}}.
\]
Let $E$ be a splitting field for $L$ over $\F_q(t)$, let $G$ be the Galois
group of $L$, and let $V$ be the space of roots
of $L$ in $E$. 

Then the function $Q:V\to E$ defined by
\[
Q(\alpha)=t\alpha^{q^{m}+q^{m-1}}+f(\alpha)
\]
is a nondegenerate $\F_q$-valued quadratic form on $V$ that is $G$-invariant. Here,
$f$ is the polynomial of degree $q^{m-1}(q^m+1)$ defined in Lemma \ref{special_polynomial}.
\end{thm}

\begin{proof}
Let $\alpha$ be any element of $V$. Then $L(\alpha)=0$ and hence
\[
t^q\alpha^{q^{m+1}}-t\alpha^{q^{m-1}}=-\Phi(\alpha).
\]
We multiply by $\alpha^{q^m}$ to obtain 
\[
t^q\alpha^{q^{m+1}+q^m}-t\alpha^{q^m+q^{m-1}}=-\alpha^{q^m}\Phi(\alpha).
\]
Given the definition of $f$, we obtain
\[
t^q\alpha^{q^{m+1}+q^m}-t\alpha^{q^m+q^{m-1}}=-f(\alpha)^q+f(\alpha).
\]
It follows that
\[
Q(\alpha)^q=Q(\alpha),
\]
and thus, since this holds for all $\alpha$, $Q$ maps $V$ into $\F_q$.

We now need to show that $Q$ is a quadratic form. We first note that as $f$ is an $\F_q$-combination of monomials of the type $x^{q^i+q^j}$, we have
\[
f(\alpha+\beta)=f(\alpha)+f(\beta)+B(\alpha, \beta),
\]
for all $\alpha$ and $\beta$ in $V$, where $B$ is a symmetric $\F_q$-bilinear mapping
of $V\times V$ into $E$, which is alternating when $q$ is even. The same argument shows that 
\[
t(\alpha+\beta)^{q^m+q^{m-1}}=t\alpha^{q^m+q^{m-1}}+t\beta^{q^m+q^{m-1}}+A(\alpha,\beta),
\]
where $A$ is also a symmetric  $\F_q$-bilinear mapping
of $V\times V$ into $E$, which is alternating when $q$ is even.

It follows that
\[
Q(\alpha+\beta)=Q(\alpha)+Q(\beta)+C(\alpha, \beta),
\]
where $C=A+B$ is also symmetric, and alternating
when $q$ is even. Since we already know that $Q$ maps $V$ into $\F_q$, it follows that $C(\alpha, \beta)\in \F_q$. This shows that $C:V\times V\to \F_q$ is a symmetric
bilinear form, alternating in the even characteristic case,  and we deduce that $Q$ is an $\F_q$-valued quadratic form on $V$.

Next, we show that $Q$ is $G$-invariant.  This is easily established. For 
$Q(\alpha)$ is an $\F_q(t)$-sum of powers of $\alpha$ and thus $\sigma(Q(\alpha))=
Q(\sigma(\alpha))$ for any $\sigma$ in $G$. But as we have shown that $Q(\alpha)$ is in
$\F_q$, 
\[
\sigma(Q(\alpha))=Q(\alpha)=Q(\sigma(\alpha)).
\]
This shows that $Q$ is $G$-invariant, as required.

Finally, we show that the polarization $C$ is nondegenerate. Let $R$ be the radical of $C$. This
is a $G$-invariant subspace of $V$ and hence is either trivial, in which case $C$ is nondegenerate, or $R=V$, by Corollary \ref{group_acts_irreducibly}. Suppose, if possible, that
$R=V$, so that $C$ is identically zero. Then we have
\[
Q(\alpha+\beta)=Q(\alpha)+Q(\beta)
\]
for all $\alpha$ and $\beta$ in $V$. Since we also have $Q(\lambda\alpha)=\lambda^2Q(\alpha)$
for all $\lambda\in\F_q$ and all $\alpha\in V$, $Q$ is a semilinear function from $V$ into
$\F_q$. Thus the kernel of $Q$, consisting of those $\alpha$ with $Q(\alpha)=0$,
is a $G$-invariant subspace of $V$ of codimension at most one. Since we know that
$G$ acts irreducibly on $V$, it follows that $Q$ is zero on $V$.

Given the definition of $Q$, we must have
\[
Q(\alpha)=t\alpha^{q^{m}+q^{m-1}}+f(\alpha)=0
\]
for all $\alpha$. Thus the polynomial $tx^{q^{m}+q^{m-1}}+f(x)$ vanishes on $V$. Since the polynomial has degree $\deg f=q^{m-1}(q^m+1)$, whereas $|V|=q^{2m}$, we have a contradiction.
Hence $Q$ is indeed nondegenerate, as required.
\end{proof}


\begin{cor} \label{type_of_orthogonal_group}
Assume the hypotheses of Theorem \ref{invariant_quadratic_form}. Then the Galois group
$G$ is a subgroup of the orthogonal group $O^-(2m,q)$ that acts transitively on the subsets
of vectors $v$ with $Q(v)=\lambda$, for all $\lambda\in\F_q$. 
\end{cor}

\begin{proof}
We have shown in Corollary \ref{factoring_the_polynomial} that $L(x)/x$ splits into
$q$ irreducible factors, of which one has degree $(q^{m-1}-1)(q^m+1)$, and $q-1$ have degree
$q^{m-1}(q^m+1)$. It follows that $G$ has exactly $q$ orbits on $V^\times$, and the
sizes of the orbits are the degrees of these irreducible factors.

Now $G$ is a subgroup of $O(2m,q)$ and the orthogonal group acts on $V^\times$ with exactly
$q$ orbits, these being the subsets of $V^\times$ consisting of vectors $v$ with
$Q(v)=\lambda$, as $\lambda$ runs over $\F_q$. It follows that the orbits of $G$
and $O(2m,q)$ are identical. We can identify the orthogonal group as $O^-(2m,q)$, since
the orbit sizes for $O^+(2m,q)$ are different (namely, one orbit of size $(q^{m-1}+1)(q^m-1)$ and $q-1$ orbits of size
$q^{m-1}(q^m-1)$).
\end{proof}

\subsection{The orthogonal group as a Galois group}

\noindent We assume that $\Phi$ is a
$q$-anti-palindromic $q$-polynomial of even $q$-degree $2m$ in $\F_q[x]$. Write
\[
\Phi=\sum_{i=0}^{2m} c_i x^{q^i}
\]
where the coefficients $c_i$ are in $\F_q$, $c_{2m}=1$, $c_i=-c_{2m-i}$ for $0\leq i\leq
m-1$, and $c_m=0$. 

We form the $q$-polynomial
\[
L=L(x)=\Phi(x)+t^qx^{q^{m+1}}-tx^{q^{m-1}}
\]
in $\F_q(t)[x]$. 
Our aim is to show that under fairly general circumstances concerning $\Phi$, the
Galois group of $L$ over $\F_q(t)$ is $O^-(2m,q)$. 

It will be convenient to deal with even $q$ and odd $q$ separately, because
for even $q$ we use transvections, whereas for odd $q$, we use reflections.

{\bf Case is $q$ even}

It turns out that when $q$ is even, the main result is surprisingly similar to Theorem \ref{symplectic_group_as_Galois_group}.

\begin{thm} \label{orthogonal_group_in_characteristic_2}
Let $q$ be a power of $2$ and 
$M(x)$ be a monic self-reciprocal polynomial of even degree $2m$, $m\geq 2$, 
in $\F_q[x]$. Suppose that $x-1$ is a factor of $M(x)$ with multiplicity two 
and all roots of $M_1(x)=M(x)/(x-1)^2$ are simple. 
Set
\[
M(x)=x^{2m}+c_{2m-1}x^{2m-1}+\cdots +c_1 x+1,
\]
where the coefficients $c_i$ satisfy $c_i=c_{2m-i}$ (and necessarily
$c_m=0$). Define the $q$-(anti)-palindromic $q$-polynomial $\Phi$ by
\[
\Phi(x)=x^{q^{2m}}+c_{2m-1}x^{q^{2m-1}}+\cdots +c_1 x^q+x
\]
in $\F_q[x]$ and similarly define $L$ by
\[
L(x)=\Phi(x)+t^qx^{q^{m+1}}+tx^{q^{m-1}}
\]
in $\F_q(t)[x]$.
Let $G$ be the Galois group  of $L(x)$ over $\F_q(t)$.
Then there is a prime ideal of degree $1$ in $\F_q[t]$
whose corresponding Frobenius class in $G$ contains  an element a power of which is 
an orthogonal transvection.
\end{thm}

\begin{proof}
Let $V$ be the space of roots of $L$ in its splitting field over $\F_q(t)$. It follows
from Theorem \ref{invariant_quadratic_form} that there is a nondegenerate $G$-invariant
quadratic form defined on $V$. 
 Corollary \ref{type_of_orthogonal_group} then implies that $G$ is a subgroup
 of $O^-(2m,q)$  that acts transitively
on the vectors of norm one. 

Let $\sigma$ be an element of the Frobenius class of $G$ defined by the specialization of $t$ to 0. The argument used in the proof of Theorem \ref{symplectic_group_as_Galois_group}
shows that there is a power of $\sigma$ that is an orthogonal transvection. 
\end{proof}

Theorem \ref{choices_for_Phi} tells us how many different choices we have for $M(x)$ and hence
for $\Phi$: they are the same as in the symplectic case in even characteristic.

\begin{cor} \label{orthogonal_group_in_characteristic_22}
Let $L(x)$ be as in Theorem \ref{orthogonal_group_in_characteristic_2}.
Then the Galois group $G$ of $L(x)$ over $\F_q(t)$ is $O^-(2m,q)$.
\end{cor}

\begin{proof}
Theorem \ref{orthogonal_group_in_characteristic_2} implies that $G$
contains an orthogonal  transvection.
Lemma \ref{is_the_whole_orthogonal_group}  implies that $G$ is $O^-(2m,q)$.
\end{proof}

{\bf Case $q$ is odd}

We consider next the odd characteristic case. Here, our construction seems to be less flexible.
Let $M(x)$ be a monic self-reciprocal polynomial of even degree $2m$ in $\F_q[x]$, where $q$ is odd. Suppose that $M(0)=-1$. Write
\[
M(x)=\sum_{i=0}^{2m}a_ix^i,
\]
where $a_{2m}=1$ and $a_{2m-i}=-a_i$ for $0\leq i\leq 2m$. 

Let $\Phi$ be the $q$-anti-palindromic $q$-polynomial in $\F_q[x]$ defined by 
\[
\Phi(x)=\sum_{i=0}^{2m}a_ix^{q^i},
\]
where the $a_i$ are the coefficients of $M$. Thus $M$ is the conventional associate of
$\Phi$. 

With this notation, we have the following theorem.

\begin{thm} \label{orthogonal_group_in_odd_characteristic}
Suppose that $-1$ is a simple root of $M$ and all other roots of $M$ in its splitting field
over $\F_q$ have odd order. Let the $q$-polynomial 
$L$ in $\F_q(t)[x]$ be defined by
\[
L(x)=\Phi(x)+t^qx^{q^{m+1}}-tx^{q^{m-1}}
\]
in $\F_q(t)[x]$. Then the Galois group $G$ of $L$ over $\F_q(t)$ is $O^-(2m,q)$. 
\end{thm}

\begin{proof}
We specialize $t$ to 0. The corresponding Frobenius class of $G$ consists of elements $\sigma$
with minimal polynomial $M(x)$ when they act on the space $V$ of roots of $L$. 
The theory of the primary decomposition shows that there is a direct sum decomposition 
$V=V_{-1}\oplus V_0$ of $V$ into $\sigma$-invariant subspaces $V_{-1}$ and $V_0$, where $V_{-1}$
is the $-1$-eigenspace of $\sigma$ and the minimal polynomial of $\sigma$ acting on
$V_0$ is $M(x)/(x+1)$. 

We know that correspondingly there is a commuting factorization
\[
\sigma=\sigma_{-1}\sigma_0=\sigma_{0}\sigma_{-1}
\]
where the minimal polynomial of $\sigma_{-1}$ is $x+1$ and the minimal polynomial of $\sigma_0$ is $M(x)/(x+1)$. 

Our hypothesis on the roots of $M$ ensures that $\sigma_0$ has odd order. This is a consequence
of Theorem 3.11 of \cite{LN}, as explained in the proof of Lemma \ref{order_prime_to_p} of this paper. It follows
that there is an odd power of $\sigma$ that equals $\sigma_{-1}$. But $\sigma_{-1}$ is a
reflection and thus $G$ contains a reflection. Since by Corollary \ref{type_of_orthogonal_group}, $G$ acts transitively
on the elements of norm one, Lemma \ref{acts_Transitively_on_reflections} shows that $G$ is
$O^-(2m,q)$.
\end{proof}

Concerning a specific choice of $M$, we may use the polynomial $(x+1)(x-1)^{2m-1}$. Then
$\sigma_0$ has order a power of $p$ and this enables the construction to be accomplished.

\section{A connection with the Weyl group}

\noindent In Subsection 3.2, we have studied  polynomials
of the form $L(x)=\Phi(x)+tx^{q^m}$, where $\Phi$ is a monic $q$-palindromic polynomial
of $q$-degree $2m$ in $\F_q[x]$. As we have seen, $\Phi$ has a conventional associate
$M(x)$, say, in $\F_q[x]$, where $M(x)$ is a monic self-reciprocal polynomial of degree $2m$.
We would like to propose the polynomial $M(x)+tx^m$ in $\F_q(t)[x]$ as a type of conventional associate of
$L(x)$. We can of course extend this concept to any $q$-polynomial over $\F_q(t)$, replacing
$x^{q^i}$ by $x^i$ and retaining the corresponding coefficients. 

Note that $M_t(x)=M(x)+tx^m$ is irreducible and it is also self-reciprocal. Let $E$ be a splitting field for $M_t(x)$ over $\F_q(t)$ and let $H$ be its Galois group. As $M_t$ is irreducible, $H$ acts
transitively on the roots of $M_t$. Furthermore, because $M_t$ is self-reciprocal, if $\alpha$ is a root of $M_t$ in $E$, so also is $\alpha^{-1}$. Thus we can arrange the roots of $M_t$ into $m$ subsets $B_1$, \dots, $B_m$, where $B_i$ consists of the roots $\alpha_i, \alpha_i^{-1}$, $1\leq i\leq m$. 

Since $H$ permutes the roots transitively, it also permutes the subsets $B_i$ transitively.
For suppose that $1\leq i,j\leq m$. Then given roots $\alpha_i$ and $\alpha_j$, we can find
$\sigma$ in $H$ with $\sigma(\alpha_i)=\alpha_j$. It follows that $\sigma(B_i)=B_j$.
The subsets $B_i$ constitute a system of imprimitivity for the action on $H$ on the roots.
There is an induced transitive permutation action of $H$ on the $B_i$ and this determines a
homomorphism from $H$ into the symmetric group $S_m$. The kernel, $N$, say, of the homomorphism
consists of those elements $\sigma$ in $H$ that satisfy
\[
\sigma(\alpha_i)=\alpha_i^{\pm 1}
\]
for $1\leq i\leq m$. It is easy to see that any such $\sigma$ satisfies $\sigma^2=1$
and we deduce that $N$ is an elementary abelian 2-group.

This discussion has shown that the Galois group $H$ of $M_t$ contains a normal elementary
abelian 2-subgroup $N$ such that $H/N$ is isomorphic to a transitive subgroup of $S_m$. We see
in particular that $|H|$ divides $2^mm!$. When $|H|=2^mm!$, $H$ is isomorphic to the
wreath product $Z_2\wr S_m$ of a group of order 2 by $S_m$. This group is also known as the Weyl
group of type $C_m$ and it is the Weyl group of $Sp(2m,q)$, which is the universal group
of type $C_m$ over $\F_q$.

We may thus see a group-theoretic connection between the Galois groups of $L(x)=\Phi(x)+tx^{q^m}$
and of $M_t(x)=M(x)+tx^m$ over $\F_q(t)$, and are led to ask if this connection is coincidental or
in any way explainable at a deeper level. One possible approach to this question is as follows.
We may take the companion matrix of $M_t(x)$ over $\F_q(t)$. This defines an element of
$GL(2m,\F_q(t))$ but as $M_t(x)$ is self-reciprocal, we may in fact consider this as an
element of $Sp(2m, \F_q(t))$, or indeed of the subgroup $Sp(2m, \F_q[t])$ (in other words
the subgroup of $Sp(2m, \F_q(t))$ whose matrix entries are polynomials in $t$). 

In the paper \cite{JKZ}, the authors investigate the principle that for a random element
in an arithmetic subgroup of a reductive group over an algebraic number field, the splitting field of the characteristic polynomial (defined, say, for the natural representation)
has Galois group isomorphic to the Weyl group of the underlying algebraic group. 
In a related result, \cite{DDS} shows that the Galois group of a generic self-reciprocal polynomial
(over an algebraic number field) is the Weyl group of the symplectic group.
In our analogy, we replace the algebraic number field by $\F_q(t)$ and employ
the arithmetic subgroup $Sp(2m, \F_q[t])$ of $Sp(2m, \F_q(t))$. The paper \cite{JKZ}
uses deep number-theoretic results to quantify its findings, whereas we have only provided a variety of examples.

Our intention in this section is to show that
if we assume that $M(x)$ has the properties described in Theorem
\ref{symplectic_group_as_Galois_group} (which guarantee that $L$ has Galois group $Sp(2m,q)$),
then the subgroup $N$ of the Galois group $H$ of $M_t(x)$ has maximum order $2^m$. Simple
examples show that this is not enough to ensure that $H$ is isomorphic to the Weyl group
of $Sp(2m,q)$. 

There are similarities in the proofs of our main results. To realize $Sp(2m,q)$ as a Galois group, we used the specialization of $t$ to 0 to produce a transvection in the Galois group
and then concluded that the Galois group contains all transvections because it acts transitively
on the nonzero roots of $L$. In the associated polynomial case $M_t(x)$, we again use the
specialization of $t$ to 0, and, by means of a theorem of van der Waerden concerning
ramification subgroups, are able to deduce that $H$ contains a transposition on the roots of $M_t$.
Then the transitive action of $H$ on the roots of $M_t$ implies that $H$ contains all transpositions on the roots. This implies that $|N|=2^m$. Thus the analogy employed is between
transvections and transpositions.

Having laid the foundations, we move towards proving our theorem of this section.
Write
\[
M(x)=(x-1)^2g_1(x)\cdots g_{r-1}(x)
\]
where the $g_i$ are different irreducible polynomials and no $g_i$ equals $x\pm 1$ (and
$M(x)$ is self-reciprocal). Let $E$ be a splitting field of $M_t(x)$ over $\F_q(t)$
and let $R$ be the integral closure of $\F_q[t]$ in $E$. Let $I$ be the prime ideal of
$\F_q[t]$ generated by $t$ and let $J$ be a prime ideal of $R$ lying over $I$. Let
$\alpha$ be a root of $M_t(x)$ in $E$ and let $K=\F_t(t)(\alpha)$ be the subfield
of $E$ generated by $\alpha$. Let $S=R\cap K$ be the integral closure of $\F_q[t]$ in $K$.

Let $H$ be the Galois group of $M_t(x)$ over $\F_q(t)$ and let $D$ be the decomposition
group with respect to $J$. Let $T$ be the inertia subgroup of $D$. As is known, $T$ is normal
in $D$, and $T$ is the kernel of the induced action of $D$ on $R/J$. Let $P=IS$ be the principal ideal of $S$ generated by $I$ and let 
\[
P=P_1^{e_1}\cdots P_l^{e_l}
\]
be the prime factorization of $P$ into powers of different prime ideals $P_i$ in $S$. The $e_i$
are positive integers, called the exponents in the factorization. Let $f_i$ be the degree
of $P_i$. This means that $S/P_i$ is a finite field of order $q^{f_i}$. We have the relation
\[
2m=e_1f_1+\cdots +e_lf_l,
\]
where $2m=\deg M_t=\deg M$.

The group $D$ acts as permutations of the roots of $M_t$. A theorem of van der Waerden, \cite{vdw}, states that under the action of $D$, the roots fall into $l$ orbits, 
$\Omega_1$, \ldots, $\Omega_l$, in accordance
with the factorization above. We may choose the notation so that $\Omega_i$
is labelled by $P_i$ and its size is $e_if_i$, $1\leq i\leq l$. Moreover, under
the action of $T$, the orbit $\Omega_i$ decomposes into $f_i$ orbits, each of size $e_i$.

Given all these preliminaries, we can proceed to the proof of a basic tool. Note that, in the context of self-reciprocal polynomials,
when we say that an automorphism acts as a transposition on the roots, we mean that the automorphism interchanges a root and its inverse, and fixes all other roots.

\begin{lemma} \label{decomposition_group_orbits}
Assume the notation introduced above, with
\[
M(x)=(x-1)^2g_1(x)\cdots g_{r-1}(x).
\]
Then $l=r$ and we may choose notation so that $e_1=2$, $f_1=1$, $e_i=1$, $f_i=\deg g_i$ for
all $i>1$. The inertia subgroup $T$ has order $2$ and is generated by an element that acts as transposition on the roots.
\end{lemma}

\begin{proof}
The prime factorization of $P$ in $S$ is determined by the factorization of $M(x)$
in $\F_q[x]$. The exponents $e_i$ are the multiplicities of the irreducible factors
of $M(x)$ and the degrees of the primes are the degrees of the corresponding irreducible
factors. Thus we may choose the notation so that $e_1=2$, $f_1=1$, all other $e_i$ are 1,
$f_i=\deg g_{i-1}$, for $i>1$. We apply van der Waerden's theorem to deduce that 
$T$ acts trivially on all orbits $B_i$ for $i>1$ and it acts as a transposition on $B_1$.

Thus if $\tau$ generates $T$, $\tau$ interchanges $\alpha_1$ and $\alpha_1^{-1}$, and fixes all
other roots. This implies that $T$ acts as stated.
\end{proof}

\begin{thm} \label{maximal_size_2^m}
Assume the notation introduced above, with
\[
M(x)=(x-1)^2g_1(x)\cdots g_{r-1}(x)
\]
and 
\[
M_t(x)=M(x)+tx^m.
\]
Let $H$ be the Galois group of $M_t(x)$ over $\F_q(t)$. Then $H$ contains
all transpositions on the roots of $M_t$ and hence contains a normal
elementary abelian $2$-subgroup $N$ of order $2^m$. 
\end{thm}

\begin{proof}
Let the roots of $M_t$ be $\alpha_1$, \dots, $\alpha_m$ and their inverses. $H$ acts transitively
on the roots since $M_t$ is irreducible in $\F_q(t)$. We may assume by Lemma \ref{decomposition_group_orbits} that $H$ contains a transposition $\tau$ that interchanges
$\alpha_1$ with its inverse. Let $i$ be greater than 1 and let $h\in H$ satisfy
$h(\alpha_i)=\alpha_1$. Then we find that $h^{-1}\tau h$ is the transposition that interchanges
$\alpha_i$ and its inverse.

It follows that $H$ contains all transpositions on the roots and thus $|N|=2^m$. The rest
is clear from previous discussion.
\end{proof}

\begin{cor} \label{wreath_product}
Assume the notation of Theorem \ref{maximal_size_2^m}. Then $H$ contains a subgroup
$H_0$ isomorphic to $H/N$ with $H_0\cap N=1$. $H$ is isomorphic to the wreath product
of a cyclic group $Z_2$ of order $2$ by $H_0$, where $H_0$ permutes $m$ copies of $Z_2$
as it permutes the $m$ subsets $B_i$ of roots.
\end{cor}

\begin{proof}
The action of $H$ on the roots of $M_t$ allows us to identify $H$ as a transitive imprimitive permutation group of degree $2m$, with $m$ blocks $B_i$ of imprimitivity of size 2. Now the subgroup $S$, say, of the symmetric group $S_{2m}$ that permutes the $B_i$ among themselves
contains a normal elementary abelian subgroup of order $2^m$ that maps each block
$B_i$ into itself. Thus, considering $H$ as a subgroup of $S$, we can clearly identify
this normal elementary abelian subgroup of $S$ with the corresponding subgroup
$N$ of $H$.

$S$ also contains a subgroup isomorphic to $S_m$, defined as follows. We may suppose that
the block $B_i$ consists of the numbers $2i-1$, $2i$, for $1\leq i\leq m$. 
Let $\sigma$  be any element of $S_m$. We extend $\sigma$ to 
an element $\sigma'$ acting on the numbers 1, 2, \dots,
$2m$ by setting
\[
\sigma'(2i-1)=2\sigma(i)-1, \quad \sigma'(2i)=2\sigma(i).
\]
The subgroup $S'$, say, of $S$ obtained in this way permutes the blocks $B_i$
as it permutes the numbers 1, 2, \dots, $m$ and is isomorphic to $S_m$. 

Clearly, $N$ intersects $S'$ trivially and $S$ is the semi-direct product $NS'$. 
The complement to $N$ in $H$ is $H_0=H\cap S'$, and this is a transitive subgroup of $S_m$.

\end{proof}

Finally, the question arises: what transitive subgroups $H_0$ of $S_m$ can we obtain
by this process? We do not have a definitive answer but we will show that some
interesting groups can be realized. Recall that we require a self-reciprocal
polynomial $M(x)$ of degree $2m$ for which 1 is a root of multiplicity two,
$-1$ is not a root, and all other roots are simple. We can construct
such a polynomial $M(x)$ in the form $x^mf(x+x^{-1})$, where $f$ is a monic polynomial
of degree $m$. 

\begin{lemma} \label{monodromy}
Suppose that in Theorem \ref{maximal_size_2^m}, we express
$M(x)$ as $x^m f(x+x^{-1})$, for an appropriate choice of $f$. Then the subgroup
$H_0$ of $H$ is isomorphic to the Galois group of $f(x)+t$ over $\F_q(t)$. 
\end{lemma}

\begin{proof}
Let the roots of $M_t(x)$ in a splitting field $E$ over $\F_q(t)$ be $\alpha_i$, $1\leq i\leq m$,
and their inverses. 
The normal subgroup $N$ of the Galois group $H$ of $M_t(x)$ maps each root $\alpha_i$ either into
itself or its inverse. We claim that the fixed field of $N$ in $E$ is generated by
the $m$ elements $\alpha_i+\alpha_i^{-1}$, $1\leq i\leq m$.

For, it is clear that $N$ fixes all such expressions. Now suppose that
an element $\sigma$ of $H$ fixes $\alpha+\alpha^{-1}$, where $\alpha$ is one of the roots. Then we have 
\[
\sigma(\alpha+\alpha^{-1})=\sigma(\alpha)+\sigma(\alpha)^{-1}=\alpha+\alpha^{-1}.
\]
Thus
\[
\sigma(\alpha)-\alpha=\alpha^{-1}-\sigma(\alpha)^{-1}=\frac{\sigma(\alpha)-\alpha}{\alpha\sigma(\alpha)}.
\]
It follows that $\sigma(\alpha)=\alpha$ or $\alpha\sigma(\alpha)=1$. We see that
$\sigma$ either fixes $\alpha$ or inverts it.

We deduce that the subgroup of $H$ fixing elementwise the subfield $F$, say, generated over
$\F_q(t)$ by the elements $\alpha_i+\alpha_i^{-1}$ is contained in $N$ and hence equals
$N$ by our observation above. 

Galois theory now tells us that $H/N$ is isomorphic to the Galois group of $F$ over
$\F_q(t)$. But $F$ is generated over $\F_q(t)$ by the elements $\alpha_i+\alpha_i^{-1}$ 
and these satisfy the polynomial equation
\[
\alpha_i^mf(\alpha_i+\alpha_{i}^{-1})=-t\alpha_i^m.
\]
It follows that $F$ is the splitting field over $\F_q(t)$ of $f(x)+t$ and
$H_0\cong H/N$ is isomorphic to the Galois group of this polynomial.

\end{proof}

The investigation of polynomials of the form $f(x)+t$, where $f(x)\in \F_q[x]$,
is quite common in field theory and the resulting Galois group is often called
the arithmetic monodromy group $f(x)+t$. We shall restrict attention to the case of even $q$ for simplicity.
In this case, we want all roots of $f$ to be simple and 0 to be a root. 

A well studied example occurs when $f(x)=x^m+x$. If $m$ is even, $f$ satisfies
the conditions above but if $m$ is odd it does not. Setting $m=2r$ for some
integer $r$, we can enquire about the Galois group of $x^{2r}+x+t$ over $\F_q(t)$,
with $q$ even. 

Magma provides information about the Galois group of such polynomials when $q$ is small
and $2r$ is reasonably small. We make use of the following result, deduced from principles
explained in \cite{L}, Chapter VI, Theorem 1.12, to push this data further.

\begin{lemma} \label{change_of_field}
Let $f(x)$ be a separable polynomial in $\F_q(t)[x]$ and let $G$ be its Galois group over
$\F_q(t)$. Let $s$ be a positive integer and let $H$ be the Galois group
of $f(x)$ considered as a polynomial over $\F_{q^s}(t)$. Then $H$ is isomorphic to
a normal subgroup of $G$ with $G/H$ cyclic of order dividing $s$.
\end{lemma}

Magma shows that the Galois group of $x^n+x+t$ over $\F_2(t)$ is $A_n$ for $n=10$
and $n=18$. It is the Mathieu group $M_{24}$ for $n=24$. Thus Lemma \ref{change_of_field}
implies that the polynomial has the same Galois group over $\F_q(t)$ for any 2-power
$q$, since the appropriate Galois groups are simple. The Galois group is $S_n$ for $n=12$, $14$ and 20. Therefore,
the Galois group of $x^n+x+t$ over $\F_q(t)$ for any 2-power $q$ is either $A_n$ or $S_n$.

These examples enable us to construct $q$-polynomials over $\F_q(t)$ with Galois
group $Sp(2n,q)$ and associated Weyl group $Z_2\wr A_n$ or $Z_2\wr S_n$,
as well as $Sp(48,q)$, with associated Weyl group $Z_2\wr M_{24}$. There is a lot of flexibility
in this construction and we have provided a few specimens.


\end{document}